\documentclass{amsart}
\usepackage{amscd,amssymb,amsmath,amsthm}
\usepackage[dvips]{graphicx}
\usepackage[all]{xy}
\theoremstyle{plain}
\newtheorem{proposition}{Proposition}

\newtheorem{lemma}[proposition]{Lemma}

\newtheorem{definition}[proposition]{Definition}
\newtheorem*{proposition*}{Proposition}
\newtheorem*{theorem*}{Theorem}
\newtheorem*{corollary*}{Corollary}
\newtheorem*{lemma*}{Lemma}
\newtheorem*{remark*}{Remark}
\newtheorem*{example*}{Example}
\newtheorem*{definition*}{Definition}

\newcommand{\Z}{\mathbb{Z}}
\newcommand{\Q}{\mathbb{Q}}
\newcommand{\R}{\mathbb{R}}
\newcommand{\C}{\mathbb{C}}

\begin{document}

\title{Frame ambiguity in Open Gromov-Witten invariants}

\author{Vito Iacovino}

\address{}

\email{vito.iacovino@gmail.com}

\date{version: \today}


\begin{abstract}
We consider Open Gromov-Witten invariants for noncompact Calabi-Yau in the case the Lagrangian has the topology of $\R^2 \times S^1$. 
The definition of the invariant involves the choice of a frame for the Lagrangian, in accord with string theory. 

Our result applies to the examples arising from Large $N$-duality. In particular it leads to knot and link invariants counting holomorphic curves. 
\end{abstract}

\maketitle

\section{Introduction}

For $X$ a \emph{compact} Calabi-Yau three-fold and $L$ a Maslov index zero Lagrangian submanifold, 
in \cite{OGW1}, \cite{OGW2}, we construct the Open Gromov-Witten invariants associated to the pair $(X,L)$. 
These invariants are the natural mathematical formulation of the Open Topological String partition function defined by Witten \cite{W2}. 
Our construction formalizes the result of Witten expressing the open topological amplitudes in terms of Wilson loop integrals of the Chern-Simons theory living on $L$. 


In \cite{AV} and \cite{AKV}, Aganagic, Klemm and Vafa consider examples of \emph{noncompact} pairs $(X,L)$ in the case that $L$ has the topology of $\R^2 \times S^1$.
In these examples it is realized that the Open Topological partition function is not uniquely determined, but it is affected by an ambiguity.
The apparence of this ambiguity is tied with the usual fact that a problem on a noncompact space is well defined only after some boundary conditions at infinity are fixed.

The computations of \cite{AV} and \cite{AKV} were made in two different ways, using mirror symmetry and using large $N$ duality. 
In mirror symmetry the ambiguity is viewed as a choice of boundary condition at infinity on the mirror non-compact brane. 
In large $N$ duality the ambiguity manifests in the choice of the frame of the knot associated to the Lagrangian submanifold.

In this paper we give a mathematical definition of Open Gromov-Witten invariants for non-compact Lagrangians $L$ whose connected components have the topology 
of $\R^2 \times S^1$. We extend the construction of \cite{OGW1} to this case.
Here we need to make the assumption that Gromov compactness holds.
We provide an explanation of the frame ambiguity entirely in symplectic topology, without using any (unproved) duality. 
Our result provides the natural definition of Open Gromov-Witten invariants used in Large $N$ duality (\cite{OV}, \cite{MV}, \cite{LMV}). 
In particular it leads to knot and link invariants counting holomorphic curves. 
In the past there have been other attempts to define such invariants in symplectic topology, 
however they lead to knot invariants which seem distinct from the original invariant (see \cite{Ng}).


We consider first the problem of defining the linking number of two curves on $L$. The linking number is defined after a choice of compactification of $L$ to a $3$-sphere.  
We introduce the topological notion of frame of $L$. 
To each frame corresponds a compactification of $L$ and therefore a definition of linking number of curves. 
We apply these considerations to define Open Gromov-Witten invariants for the pair $(X,L)$ 
as generalized linking number of a system of chains following the same recipe of \cite{OGW1}. 
An analogous construction can be applied every time that $H_2(L, \Q)=0$. 

We can also understand the appearance of the frame ambiguity using Chern-Simons theory, more in relation with \cite{W2} (see \cite{OGW2}). 
To define the open Gromov-Witten invariants for the pair $(X,L)$ we need in particular to define the Chern-Simons propagator living on $L$.
As usual in not compact situation, the problem involves the choice of boundary conditions at infinity.
A way to proceed consist in define the propagator on $L $ restricting the propagator defined on a compactification of $L$. 



\section{Frames and Linking numbers} 

Let $L$ be a three manifold with the topology of $\R^2 \times S^1$. In this section we consider the problem of define the linking number of two curves on $L$. 
We will show that (essentially because $H_2(L)=0$) it is possible to extend the usual definition of linking number on $S^3$ to this topology 
but it can not be done in a unique way.

Let $L' \subset L$ be a complement of compact set of $L$ such that 
$$ L' \cong \R \times  S^1 \times S^1 .$$
In particular $H_1(L') \cong \Z \times \Z$.
Fix a generator $v$ of the kernel of $H_1(L') \rightarrow H_1(L)$:
$$\text{Span}_{\Z} \{  v \} = \text{ker} \{ H_1(L') \rightarrow H_1(L) \} .$$
\begin{definition} \label{def-frame}
A frame of $L$ is an element $f \in H_1(L') $ such that 
\begin{equation} \label{generate}
\text{Span}_{\Z} \{  v, f \}  = H_1(L') .
\end{equation}
We identify two frames if they differ by a sign.
\end{definition}

Fix an isomorphism $H_1(L') \cong \Z \times \Z$ and write $f= (f_1,f_2)$ and $v=(v_1,v_2)$. Condition (\ref{generate}) is equivalent to 
\begin{equation} \label{condition}
f_1 v_2 - f_2 v_1 = \pm 1  .
\end{equation}  

We now consider the problem of define linking numbers of curves on $L$ after we pick a frame of $L$. 
By a linking number of curves on $L$ we mean a law that associates to every pair of embedded disjoint curves on $L$ 
an integer that jumps by $\pm 1$ when two curves cross, with sign determinate according to the sign of the cross. 
\begin{lemma} \label{frame-link}
To each frame of $L$ is associated a definition of linking number of curves on $L$.
\end{lemma}
\begin{proof} 
The linking number between two curves on $S^3$ is defined by shrinking the curves and counting intersection numbers.

Not all the curves in $L$ can be shrinked. 
However we can adapt the definition of $S^3$ declaring zero curves at infinity representing some multiple of $f$ in $H_1(L')$. 
More precisely we can proceed as follows. 

Two curves $\gamma_1^0$ and $\gamma_2^0$ on $L'$ representing a multiple of the homological class $f$ are defined unlinked if 
for each compact set $K$ of $L$ there exists an homotopy that does not intersect $\gamma_2^0$, of $\gamma_1^0$ to a curve on $L \setminus K$ , 
and there exists a homotopy of $\gamma_2^0$ that does not intersect $\gamma_1^0$, to a curve in $L \setminus K$.

Observe that there are plenty of pairs of curves with this property. 
For example this holds for each pair $(\gamma_1^0, \gamma_2^0 )$ with 
$\gamma_1^0$ in $\{ t_1 \} \times S^1 \times S^1 $ and $\gamma_2^0$ in $\{ t_2 \} \times S^1 \times S^1 $ for $t_1 \neq t_2$
and representing a multiple of the homological class $f$
(here we fixed an isomorphism $L' \cong \R \times S^1 \times S^1$).

Fix a pair $(\gamma_1^0, \gamma_2^0 )$ of curves on $L'$ representing the homological class $f$ and unlinked in the sense above. 
Given two curves $\gamma_1$ and $\gamma_2$ on $L$, there exist $c_1, c_2 \in Z$ such that $ [\gamma_1] = c_1 f $ and $ [\gamma_2] = c_2 f $ in $H_1(L)$. 
For $i=1,2$, $\gamma_i$ is homotopic to $c_i \gamma_i^0$, that is there exist a singular chain $B_i \in C_2([0,1] \times L)$ such that 
$$ \partial B_i = c_i \{ 1 \} \times \gamma_i^0 - \{ 0 \} \times \gamma_i .$$
Define the linking number of $ \gamma_1, \gamma_2$ as the intersection number of $B_1$ and $B_2$
$$ \text{link}(\gamma_1, \gamma_2) =(B_1 \times_{[0,1]} B_2 )\cap \Delta=  B_1 \cap B_2 .$$
From the definition it follows that $ \text{link}(\gamma_1, \gamma_2)$ does not depend on the choice of the pair $( \gamma_1^0 , \gamma_2^0)$.
In fact if $( \gamma_1^0 , \gamma_2^0)$ and $( \gamma_1^1 , \gamma_2^1)$ are two pairs of unlinked curves in the sense above, 
there exists a homotopy between $\gamma_1^0$ and $\gamma_1^1$ and a homotopy between $\gamma_2^0$ and $\gamma_2^1$ not intersecting each other.

\end{proof}

We can also see the correspondence between frames and linking numbers considering compactifications of $L$ to spheres. 
Observe that, if we have a compactification of $L$ to an $S^3$ 
\begin{equation} \label{embedding}
\bar{f}: L \hookrightarrow S^3 .
\end{equation} 
we can define the linking number of curves on $L$ considering them as curves on $S^3$.
Of course the result will depend on the compactification. 

The connection with the preview definition is the following. Frames of $L$ induces compactifications of $L$
$$ \text{frames of }L \leadsto \text{compactifications of }L   .$$
To a frame $f$ it corresponds a campactification (\ref{embedding}) such that $S^3 \setminus L \cong S^1$ and 
\begin{equation} \label{kernel}
\text{Span}_{\Z} \{ f \} = \text{ker} (H_1(L') \rightarrow H_1(S^3 \setminus K)).
\end{equation}
Fix a diffeomorphism $L \cong \R^2 \times S^1$. 
The compactification $S^3$ is made gluing to $L$ another copy of $\R^2 \times S^1$ over $\R^2 \backslash \{ (0,0) \})$ in the following way.
Observe that $(\R^2 \backslash \{ (0,0) \}) \times S^1 \cong \R \times T^2$, where $T^2$ is the torus of dimension two. 
The elements $f, v \in H_1(L')$ define a basis (over $\Z$) of $H_1(T^2)$. 
Let $A_f : T^2 \rightarrow T^2$ be a map that interchange $f$ and $v$ up to sign and preserves the orientation. Define 
$$ S^3= L \sqcup_g  (\R^2 \times S^1) $$
where $g$ is the diffeomorphism $g : (\R^2 \backslash \{ (0,0) \}) \times S^1  \rightarrow(\R^2 \backslash \{ (0,0) \}) \times S^1 $ given by 
\begin{equation} \label{diffeo}
(r,(s,t)) \rightarrow (-r, A_f( s, t))
\end{equation}
where $r \in \R$ and $s,t \in S^1 $.
It is easy to see that the definition of linking number of Lemma (\ref{frame-link}) agrees with the one coming from (\ref{embedding}).




\section{Open Gromov-Witten invariants} 

Let $X$ be a noncompact Calabi-Yau three-folds and $L$ a Lagrangian submanifold with topology of $\R^2 \times S^1$ and Maslov index zero.
In this section we want to define the Open Gromov Witten invariants of the pair $(X,L)$ following the construction of \cite{OGW1}.

Since the space is noncompact we need to make the technical assumption that Gromov compactness holds for closed holomorphic curves 
as well as for holomorphic curves with boundary mapped on $L$. 
The problem of Gromov compactness for noncompact manifolds is considered for example in \cite{Sik}. 

As in \cite{OGW1}, in order to solve the problem of the bubbling of spheres from constant disks, we assume that $L$ is homological trivial in $X$. 
Fix a $B \in C_4(X)$ such that
\begin{equation} \label{B}
 \partial B =L  .
\end{equation}
The necessary modifications are the same of \cite{OGW1} and they will not be reported here.

\subsection{Multi-curves}
The set of decorated graphs $\mathcal{G}$ is defined as in \cite{OGW1}.
For each $G \in \mathcal{G}$, define the moduli space of multi-curves $\overline{\mathcal{M}}_G$ as in \cite{OGW1} 
$$ \overline{\mathcal{M}}_G =  \left( \prod_{i \in I} \overline{\mathcal{M}}_i  \right) / \text{Aut}(G) . $$

In order to be able to count the winding number of the boundary components of the multi-curves 
we need to refine the definition of evaluation map on the boundary marked points.
For this we need to encode the homotopy class of the boundary paths between consecutive boundary marked points as follows.

Let $ \{x_e \}_{e \in H(G)}  \in L^{H(G)} $.
For $v \in V(G)$ let $H(v)= \{ e_1,..,e_{|H(v)|} \}$ 
be the cyclic order of the half edges starting in $v$. 
Define
$$ \pi_1(v) =  \prod_i \pi_1(L, x_{e_i}, x_{e_{i+1}}) $$
where, for each $i$, 
$$ \pi_1(L, x_{e_i}, x_{e_{i+1}})  = \{ \text{homotopy classes of paths between $x_{e_i}$ and $x_{e_{i+1}}$}\} .$$
In the particular case that $H(v)$ is empty we set $\pi_1(v) =  \pi_1(L)$. 

Let $\tilde{L}^{H(G)}$ be the covering space of $L^{H(G)}$
\begin{equation} \label{covering}
 \pi : \tilde{L}^{H(G)} \rightarrow  L^{H(G)}  
\end{equation}
with fibers
\begin{equation} \label{fiber}
\pi^{-1}( \{x_e \}_{e \in H(G)}) =  \prod_{v \in V(G)}  \pi_1(v).
\end{equation}


Define
$$ \tilde{L}_G = \tilde{L}^{H(G)} / \text{Aut}(G)  .$$
The evaluation map on the boundary punctures can be refined considering the boundary paths between consecutive punctures. This leads to the natural map 
\begin{equation} \label{evaluation}
\text{ev} : \overline{\mathcal{M}}_G \rightarrow \tilde{L}_G .
\end{equation}

\subsection{Systems of singular chains}

For $e$ internal edge of $G$, there exists a natural lifting $\tilde{\Delta}_e  $ of the big diagonal $\Delta_e$ of $L_G$ 
to a smooth simplicial chain of the orbifold $ \tilde{L}_G$.


Consider $e$ connecting two vertices $v_0$ and $v_1$ (it can be $v_0 = v_1$).
Let $ e_0 , e_1$ the puncture associated to $e$, in the boundary components associated to $v_0, v_1$. 
Let $e_i^+$, $e_i^-$ be the half edges after and before $e_i$ respectively. 
On $\tilde{\Delta}_e$, since $x_{e_0}= x_{e_1}$, we can define the gluing maps
$$\pi_1(L, x_{e_0^-}, x_{e_0}) \times \pi_1(L, x_{e_1}, x_{e_1^+}) \rightarrow \pi_1(L, x_{e_0^-}, x_{e_1^+})  $$
$$\pi_1(L, x_{e_1^-}, x_{e_1}) \times \pi_1(L, x_{e_0}, x_{e_0^+}) \rightarrow \pi_1(L, x_{e_1^-}, x_{e_0^+}) . $$
The cases where at least one of $v_0$ and $v_1$ have valence equal to one are specials.
If for example $v_0$ has valence one, the gluing map is
$$\pi_1(L, x_{e_1^-}, x_{e_1}) \times \pi_1(L, x_{e_0}, x_{e_0})  \times \pi_1(L, x_{e_1}, x_{e_1^+}) \rightarrow \pi_1(L, x_{e_1^-}, x_{e_1^+}) . $$

The gluing maps above lead to the map 
\begin{equation} \label{delta-projection}
  \tilde{\Delta}_e \rightarrow \tilde{L}_{G/e} .
\end{equation}

We use this map to extend the definition of system of chains.

\begin{definition} \label{def-system}
A system of singular chains 
$$W_{\mathcal{G}} = \{  W_G  \}_{G \in \mathcal{G}} $$ 
is a collection of singular chains
$W_G \in C_{|E(G)|}( \tilde{L}_G , \mathfrak{o}_G) $
with twisted coefficients in $\mathfrak{o}_G$.   
Fulfilling the following properties:
\begin{itemize}
\item[(a)] For each $G \in \mathcal{G}$, and $e$ an internal edge of $G$, $W_G$ intersects $\tilde{\Delta}_e$ transversely. 

Let $\partial_e W_G$ be the singular chain on $\tilde{L}_{G/e}$ defined by  
$$  W_G \cap \tilde{\Delta}_e \in C_*(\tilde{L}_{G/e},\mathfrak{o}_{G/e} ) $$
using the map (\ref{delta-projection}). 

\item[(b)] The following identity holds as currents  
\begin{equation} \label{boundary-collection}
\partial W_G = \sum_{G'/e'=G} \partial_{e'} W_{G'} .
\end{equation}
Here the sum is taken over all the pairs $(G',e')$ with $e' \in E(G')$ such that $G'/e' \cong G$.
\end{itemize}
\end{definition}

In the same way we can extend the definition of homotopies of system of chains and equivalences of homotopies.

Exactly as in \cite{OGW1} we can constrain the pertubation of the Kuranishi spaces $ \{ \overline{\mathcal{M}}_G \}_{G \in \mathcal{G}}  $.
Using these perturbations and the map (\ref{evaluation}) we define a system of chains 
\begin{equation} \label{system}
W_{\mathcal{G}} = \{  W_G \}_{G \in \mathcal{G}}  
\end{equation}
with $W_G \in C_*(\tilde{L}_G , \mathfrak{o}_G)$.

As in \cite{OGW1} we have 
\begin{proposition} \label{higher-invariance}
(\ref{system}) defines a system of chains.
Different choices of perturbations lead to homotopic systems of chains, with homotopy uniquely determined up to equivalence.
\end{proposition}

Now we can define the Open Gromov-Witten invariants.

\begin{proposition}
Let $f$ be a frame of $L$. 
The rational numbers $F_{g,n_1...,n_h}^f(A)$ are defined counting multi-curves in the relative homology class $A \in H_2(X,L)$ of genus $g$, 
$h$ boundary components and with the $i$-th boundary component winding $n_i$ times around the nontrivial cycle of $L$. 
These numbers depend only on the pair $(X,L)$ and $f$ (and the choice of $B$ in (\ref{B})).

\end{proposition}

\begin{proof}
As in \cite{OGW1} the rational numbers $F_{g,n_1...,n_h}^f(A)$ are defined in terms of generalized linking numbers for the system of chains (\ref{system}). 

Recall that to define the generalized linking number of a system of chain in \cite{OGW1} 
we recursively contract the chains starting with the component associated to the graph with a maximal number of edges. 

In our case the singular chains are not homologically trivial. 
However we can adapt the procedure of Lemma (\ref{frame-link}) to get an invariant depending of the choice of the frame $f$.

The models of unlinked system of chains depends on the choice of some data. 
For each decorated graph $G$ choice real numbers
\begin{equation} \label{data}
 \{  t_e \}_{e \in H(G)} .
\end{equation}
Given a vertex $v_0$ of $G$ associated to a disk of area zero, and a real number $s$, define a new (\ref{data})
by changing $t_e$ to $t_e +s$ for each $e \in H(v_0)$. 
This generate an equivalence relation on the set of data (\ref{data}).    


We require that the data (\ref{data}) satisfies the following compatibility condition 
for each $e_0  \in H(G)$ oriented edge ending to a vertex $v_0$ associated to the boundary of a disk of area zero.
Let $e_1$ be $e_0$ with opposite orientation. 
For $e \in H(G)$ with $e \neq e_0$, so that $e$ defines an edge $\tilde{e}$ of $G/{e_0}$.
If $e$ start from $v_0$, we require $t_{\tilde{e}}= t_e - t_{e_1}  + t_{e_0}$.
Otherwise we require $t_{\tilde{e}}= t_e$.   


Using a generic choice of data (\ref{data}) with the above compatibility condition we now construct a system of chains.
We first define a singular chain $W_v \in C_*(L^{H(v)})$ for each vertex $v$ of a graph $G$. 

Let $e_1, e_2, ..., e_k$ be a cyclic order of the edges starting from $v$. 
Suppose first that $v$ corresponds to the boundary component of a curve of positive area. As in Lemma (\ref{frame-link}) 
let $L' \subset L $ be a complement of a compact subset of $L$ such that $L_0 \cong \R \times S_1 \times S_1$.
Fix a curve $\gamma_0 : S_1 \rightarrow S_1 \times S_1$ representing the frame $f$. 
Let $W_v$ be represented by the geometric cycle
$$(s_{e_1} , s_{e_2} , ..., s_{e_k} ) \rightarrow ((t_{e_1} , \gamma_0(s_{e_1} )), ((t_{e_2} , \gamma_0(s_{e_2} ),..., (t_{e_k} , \gamma_0(s_{e_k} )) $$

Assume now that $v$ is associated to the boundary of a disk of area zero.
The geometric cycle 
$$\R^+ \times S_1 \times S_1 \rightarrow (\R \times S_1 \times S_1)^k \subset L^k $$
$$(t, p) \rightarrow (t, p), (t_{e_1} - t_{e_2} + t, p), ...,(t_{e_k} -  t_{e_{k-1}} + t, p),( t_{e_1} - t_{e_k} + t, p)$$
can be extended to $L$. Let $W_v$ be the corresponding singular chain.

Define 
$$ W_G^{unlinked} = \prod_{v \in V(G)} W_v .$$

Using an iterative argument as in \cite{OGW1} it is not hard to show that any system of chain is homotopic to a unique  system of chain $W_{\mathcal{G}}^0$ 
such that, for each $G$, $W_G^0$ is some rational number times $W_G^{unlinked}$
$$  W_G^0 \sim W_G^{unlinked} .$$ 

Different choices of (\ref{data}) lead to different $W_{\mathcal{G}}^0$. 
However the $0$-chains associated to graphs without edges do not change.

The construction above applied to the system of chain (\ref{system}) leads therefore to invariants of the triple $(X,L,f)$. 
We define Open Gromov-Witten invariants $F_{g,n_1...,n_h}^f(A)$ by the formula
$$  W_{G_{g,h}}^0 = \sum  F_{g,n_1,...,n_h}^f(A) [n_1,...,n_h]$$
where for each $(g; h)$ let $G_{g,h}$ be the unique decorated graph with no edges of genus $g$ and $h$ boundary components and
$[n_1,...,n_h]$ is the singular $0$-chain with support in the point $(n_1,...,n_h) \in \tilde{L}_{G_{g,h}} $ 
(observe that $\tilde{L}_{G_{g,h}}=  \pi_1(L)^h \cong \Z^h $). 

    
\end{proof}

It is straightforward to extend the preview result in the case that $L$ is a disjoint union of submanifolds
$$ L = \sqcup_{\alpha=1}^k L_{\alpha} $$
where each component $L_i$ has the topology of $S^1 \times \R^2$. 
In this case we assume that each component $L_i$ is homological trivial. For each $\alpha$ let $f_{\alpha}$ be a frame for $L_{\alpha}$. As before we can define the invariants
$$  F_{g, \vec{n}^{(1)},...,  \vec{n}^{(k)}}^{ f_1,..,f_k }(A).$$

Under the weaker assumption that $L$ is homological trivial but not all its single components $L_{\alpha}$, 
it is not possible to distinguish on which connected component of $L$ the boundary components with zero winding number are mapped. 





\section{Knots and Links invariants} 

Gopakumar and Vafa in \cite{GV1} and \cite{GV2} conjectured that the closed topological strings amplitudes (closed Gromov-Witten invariants) of the resolved conifold 
$$X=\mathcal{O}(-1) \oplus \mathcal{O}(-1) $$ 
can be computed in terms of the Chern-Simons invariants of $S^3$ (here $\mathcal{O}(-1)$ is the tautological bundle over $\C P^1$). This duality is referred to as Large $N$ duality. 
The conjecture was later extended to Open amplitudes in \cite{OV}, \cite{MV}, \cite{LMV}. The result in particular predicts the existence of knot and link invariants counting holomorphic curves. In this section we apply the result of the previous section to give a mathematical definition of these invariants.   

Let $\mathcal{K}$ be a knot in $S^3$ represented by the curve
$$ \gamma : S^1 \rightarrow S^3 . $$ 
Large $N$ duality associate to $\mathcal{K}$ a Lagrangian submanifold of the resolved conifold $\mathcal{O}(-1) \oplus \mathcal{O}(-1) $
$$ \mathcal{K} \leadsto L_{\mathcal{K}}  $$
such that the topological open string invariants of $L_{\mathcal{K}}$ can be computed in terms of Chern-Simons invariants of $\mathcal{K}$. 
The Lagrangian $L_{\mathcal{K}}$ was constructed for the unknot in \cite{OV} and generalized to all knots in \cite{Taubes} and \cite{Kos}. For our purpose it is important to notice that as a manifold $L_{\mathcal{K}}$ is the same of the conormal Lagragian $\mathcal{C}_{\mathcal{K}} \subset T^*S^3$ of $\mathcal{K}$ (see \cite{Kos}).


\begin{lemma}
Each frame of the knot $\mathcal{K}$ defines a frame of the Lagrangian $L_{\mathcal{K}} $. This correspondence is one to one
\begin{equation} \label{correspondence0}
\text{frames of }\mathcal{K} \leftrightarrow \text{frames of }L_{\mathcal{K}}
\end{equation}
\end{lemma}
\begin{proof}
Let 
$$f : S^1 \rightarrow  \gamma^*(TS^3)/ \langle \dot{\gamma}  \rangle  $$
be a frame of $\mathcal{K}$.

Since $L_{\mathcal{K}}$ and $\mathcal{C}_{\mathcal{K}}$ are the same as manifolds, we need to associate to $f$ a frame of $\mathcal{C}_{\mathcal{K}}$.
Define 
$$\tilde{\mathcal{C}}_{\mathcal{K}} = \mathcal{C}_{\mathcal{K}} \backslash K $$
where $K$ is compact with the topology of $S^1 \times Q$ with $Q$ contractible (as in Definition \ref{def-frame}). We need to associate to $f$ an element of $ H_1(\tilde{\mathcal{C}}_{\mathcal{K}})$ such that condition (\ref{condition}) holds.  

Let $\eta$ be a section
$$\eta: S^1 \rightarrow \gamma^*(T^*S^3) $$
such that for each $s \in S^1$ 
$$ \langle \eta(s) , \dot{\gamma} (s) \rangle =0   $$
$$ \langle \eta(s) , f (s) \rangle =0   . $$
The graph of $\eta$ defines a one cycle of $\tilde{\mathcal{C}}_{\mathcal{K}}$ that satisfies condition (\ref{condition}), and therefore a frame of $\mathcal{C}_{\mathcal{K}}$.




\end{proof}

In \cite{Kos} it is proved that Gromov compactness holds for the pair $(X, L_{\mathcal{K}})$ for each knot $\mathcal{K}$. Moreover since $H_3(X) =0 $ the Lagrangians $L_{\mathcal{K}}$ is homological trivial. 

\begin{proposition}
Let $\mathcal{K}^f$ be a framed knot in $S^3$. The Open Gromov-Witten invariants $F_{g,(l_1,...,l_h)}^ f$ of genus $g$ and winding numbers $l_1,...,l_h $ are defined.
\end{proposition}

Observe that since $H_4(X)=0$, the invariants do not depend on the choice of $B$ in (\ref{B}). 

It is straightforward extend the result to the case of link invariants. 
\begin{proposition}
Let $\mathcal{K}_1^{f_1},...,\mathcal{K}_k^{f_k}$ be a framed link in $S^3$. The Open Gromov-Witten invariants $ F_{g,(\vec{n}^{(1)},...,\vec{n}^{(k)})}^{  f_1, ...,f_k} $ 
of genus $g$ and winding numbers $(\vec{n}^{(1)},...,\vec{n}^{(k)})$ are defined.
\end{proposition}


\subsection{Katz-Liu computation} Katz and Liu in \cite{KL} computed the Open Gromov-Witten invariants for the Lagrangian of Large $N$ duality associated to the unknot. In this particular case the Lagrangian is invariant by the $S^1$-actions of the resolved conifold. The computation of Katz and Liu was based on the assumption that the virtual localization technique could be applied to Open Gromov-Witten invariants as in the closed case, without having a definition of Open Gromov-Witten invariant. It was realized that the result of the computation depends on the choice of the weight of the $S^1$-action. Their result was in accord with the prediction of string theory after identifying the weight of the $S^1$ action with the frame of the knot.

Liu in \cite{Liu} defines an invariant counting open curves associated to a triple $(X,L,\rho)$, where $\rho$ is an admissible $S^1$ action on $X$ (this means that $\rho$ preserves $L$ and acts freely on $L$). This invariant could be computed using the localization formula with respect to the $S^1$-action $\rho$. In particular it gives a rigorous definition of the invariant computed in \cite{KL}. 






%
An admissible $S^1$ action defines a frame of the Lagrangian (in the sense of Definition \ref{def-frame}): 
the associated element of $H_1(L')$ is the cycle given by an orbit of the $S^1$-action. 
In the example of \cite{KL} we have then a one to one correspondence
\begin{equation} \label{correspondence}
\text{admissible $S^1$-actions} \leftrightarrow \text{frames of }L
\end{equation}
The frame (\ref{correspondence}) is precisely the one for which the linking numbers of the orbits of the $S^1$-action are zero.  
\begin{lemma} 
The invariant of Liu \cite{Liu} agrees with the Open Gromov-Witten invariants $F_{g,\overrightarrow{n}}(A)^f$ under the correspondence (\ref{correspondence}).  
\end{lemma}
The Lemma follows from the fact that the contribution of nontrivial (that is with at least an edge) 
decorated graphs is zero if we use of a perturbations invariants with respect to the inducted $S^1$-action (as in \cite{Liu}). 
We do not prove these statement, however it can be quickly understood if we assume that the localization technique can be applied to our case. 
The contribution of nontrivial graphs is computed in terms of linking numbers of boundaries components of curves. 
Since these components are orbits of the $S^1$-action, for the choice of the frame (\ref{correspondence}) their linking number is zero as observed before.

\end{document}